\documentclass{amsart}

\usepackage{amssymb,amsthm}
\usepackage{graphicx}

\newcommand{\nuke}{\mathcal{N}}
\newcommand{\img}[1]{\mathop{\mathrm{IMG}}\left(#1\right)}
\newcommand{\arr}{\rightarrow}
\newcommand{\alb}{\mathsf{X}}

\newcommand{\xs}{\alb^*}
\newcommand{\xo}{\alb^\omega}

\newcommand{\autxs}{\mathop{\mathrm{Aut}}(\xs)}
\newcommand{\C}{\mathbb{C}}

\newcommand{\gw}{\mathcal{D}}

\newcommand{\A}{\mathrm{A}}
\newcommand{\B}{\mathrm{B}}
\newcommand{\G}{\Gamma}
\newcommand{\wt}[1]{\widetilde{#1}}
\newcommand{\M}{\mathcal{M}}

\newcommand{\rist}{\mathsf{RiSt}}

\newtheorem{theorem}{Theorem}[section]
\newtheorem{proposition}[theorem]{Proposition}
\newtheorem{corollary}[theorem]{Corollary}
\newtheorem{lemma}[theorem]{Lemma}

\theoremstyle{definition}
\newtheorem{defi}{Definition}
\title[An uncountable family with isomorphic profinite completions]{An
  uncountable family of 3-generated groups with isomorphic profinite completions}

\author{Volodymyr Nekrashevych}

\begin{document}

\begin{abstract}
We construct an uncountable family of 3-generated residually finite
just-infinite groups with isomorphic profinite completions. We also
show that word growth rate is not a profinite property.
\end{abstract}

\maketitle

\section{Introduction}

Profinite completion $\widehat G$ of a group $G$ is the limit
of the inverse system of all finite quotients $G/N$
with respect to canonical epimorphisms
$G/N_1\arr G/N_2$ induced by inclusions $N_1\subseteq N_2$.

If $K$ is the intersection
of all normal subgroups of finite index of $G$, then $\widehat{G/H}$
is naturally isomorphic to $\widehat G$. Therefore, we may restrict
ourselves to \emph{residually finite} groups, i.e., assume that intersection of
finite index normal subgroups of $G$ is trivial.

To what extend does the profinite completion
$\widehat G$ determine the structure of $G$? Which group-theoretic properties are
preserved if we replace a residually finite
group by a residually finite group with the same profinite completion?
(Such properties are said to be \emph{profinite}.)

One of motivations of these question comes from the
paper~\cite{grothendieck:profinite} of A.~Grothendieck, where the
following questions was asked in relation with representation theory
of groups. Let $u:G\arr H$ be a homomorphism of finitely presented
residually finite groups such that the induced homomorphism $\widehat
u:\widehat G\arr\widehat H$ is an isomorphism. Is $u$ an isomorphism?
This question was negatively answered by M.~Bridson and F.~Grunewald
in~\cite{bridsongrunewald:groth}. Finitely generated examples (without
the condition of being finitely presented) were constructed before
in~\cite{platonovtavgen,basslubotzky:nonarithmetic,pyber:intermediate}.

The ``flexibility'' of a group $G$ in the sense of its relation to the
profinite completion is formalized by the notion of its
\emph{genus}. It is known that two
finitely generated groups have isomorphic completions if and only if
the sets of their finite quotients are equal,
see~\cite[Corollary~3.2.8]{ribeszalesskii:book}.
The \emph{genus} of a group $G$ is the set of isomorphism classes of
residually finite finitely generated groups $H$ (or groups $H$ belonging
to some other fixed class) such that the sets of
isomorphism classes of finite quotients of $G$ and $H$ are equal,
see~\cite{pickel:nilpotent,grunewaldzalesski}. Groups whose genus contains
only one element are completely determined by their profinite
completion. Groups that have finite genus also can be considered as
``rigid''.

It was shown by P.~F.~Pickel that genus of a virtually nilpotent finite
generated group is always finite~\cite{pickel:virtnilpotent}. Later it was shown by
F.~J.~Grunewald, P.~F.~Pickel, and D.~Segal~\cite{gps:polyciclic} that the same
result holds for virtually polyciclic groups. See a survey of similar
rigidity results in~\cite{grunewaldzalesski}.

Examples of infinite (countable) genera for metabelian groups
was given by P.~F.~Pickel in~\cite{pickel:metabelian}.

L.~Pyber constructed examples of groups of uncountable genus in~\cite{pyber:intermediate} using direct products of alternating groups.

In our note we construct a new family of finitely generated
groups of uncountable genus. Namely, we construct an uncountable
family of 3-generated groups $\gw_w$, $w\in\{0, 1\}^\infty$
with the following properties.
\begin{enumerate}
\item The isomorphism classes in the family $\{\gw_w\}$ are countable,
  hence there are uncountably many different isomorphism classes among
  the groups $\gw_w$. 
\item Each group $\gw_w$ is residually finite, and every proper
  quotient of $\gw_w$ is a finite 2-group.
\item The profinite completions of $\gw_w$ are pairwise isomorphic for
  all sequences $w\in\{0, 1\}^\infty$ that have infinitely many zeros.
\item There are uncountably many different word-growth types among the
  groups $\gw_w$.
\end{enumerate}

The family $\{\gw_w\}$ appears naturally in the study of group-theoretic
properties of iterations of the complex polynomial 
$z^2+i$ and the associated map on the moduli space of
4-punctured sphere, see~\cite{bartnek:rabbit}. It was for the first time
defined in~\cite{nek:ssfamilies}. Later, in~\cite{nek:nonuniform}, it
was shown that one of the groups in this family has non-uniform
exponential growth, and then it was used in~\cite{nek:dendrites} to
study the structure of the Julia set of an endomorphism of $\mathbb{CP}^2$.

We have tried to make these notes reasonably self-contained. 
Proofs of all results of~\cite{nek:ssfamilies} necessary for the main result (existence of three-generated
residually finite groups of uncountable genus) are included, except
for some facts that require straightforward computation. 

\subsection*{Acknowledgements} The author is grateful to 
A.~Myasnikov, L.~Pyber, 
A.~Reid, M.~Sapir, D.~Segal, and P.~Zalesski for discussions and remarks
on the topics of these notes. The author was supported by NSF grant DMS1006280.

\section{Definition of the family and its properties}

\subsection{Topological definition}
\label{ss:polynomials}

Let $f_1, f_2, \ldots$ be a sequence of complex polynomials, which we
arrange in an inverse sequence
\[\C\stackrel{f_1}{\longleftarrow}\C\stackrel{f_2}{\longleftarrow}\C\stackrel{f_3}{\longleftarrow}\cdots\]

Suppose that
for every $n\ge 0$ there exists a set $\{\A_n, \B_n, \G_n\}\subset\C$
of pairwise different numbers such that $\A_{n-1}$ is the critical value of $f_n$,
and
\[f_n(\A_n)=\B_{n-1},\quad f_n(\B_n)=\G_{n-1},\quad f_n(\G_n)=\B_{n-1},\]
for all $n\ge 1$, see Figure~\ref{fig:covering}. 
An example of such a sequence is the constant
sequence $f_n(z)=z^2+i$, for $\A_n=i$, $\B_n=i-1$, $\G_n=-i$.

\begin{figure}
\centering
\includegraphics{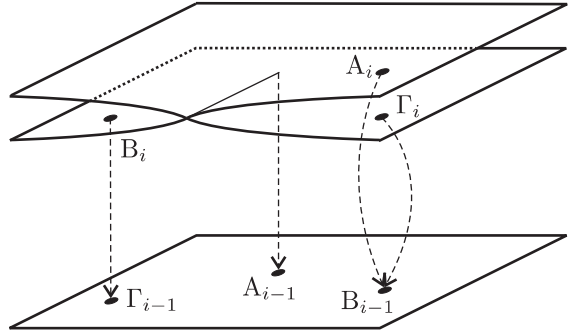}
\caption{Polynomials $f_i$}
\label{fig:covering}
\end{figure}

Denote $\M_n=\C\setminus\{\A_n, \B_n, \G_n\}$. The restriction $f_n:\M_n'\arr\M_{n-1}$, where
$\M_n'=f_n^{-1}(\M_{n-1})\subset\M_n$, is a degree two covering map.

Choose a basepoint $t\in\M_0$. The union of its backward images:
\[T_t=\{t\}\sqcup\bigsqcup_{n\ge 1}(f_1\circ f_2\circ\cdots\circ f_n)^{-1}(t)\]
has a natural structure of a rooted tree with the root $t$ in which a vertex
$z\in (f_1\circ f_2\circ\cdots\circ f_n)^{-1}(t)$ is connected to $f_n(z)$.

The fundamental group $\pi_1(\M_0, t)$ acts naturally
by the monodromy action on the levels $L_n=(f_1\circ\cdots\circ f_n)^{-1}(t)$ 
of the tree $T_t$. These actions are automorphisms of $T_t$. Denote by
$\img{f_1, f_2, \ldots}$ be the quotient of $\pi_1(\M_0, t)$ by the
kernel of the action. See more on the iterated monodromy groups of
this type in~\cite{nek:polynom}.

\subsection{Automorphisms of a rooted tree} Let us
give a more explicit description of the groups
$\img{f_1, f_2, \ldots}$ as groups acting on rooted
trees.

Let $\alb$ be a finite alphabet. Denote by $\xs$ the tree of finite
words over $\alb$. Two vertices of this tree are connected by an edge
if and only if they are of the form $v, vx$, for $v\in\xs$ and
$x\in\alb$. The root of the tree $\xs$ is the empty word $\emptyset$.

Denote by $\autxs$ the automorphism group of the rooted tree
$\xs$. Let $g\in\autxs$, and let $\pi\in\mathrm{Symm}(\alb)$ be the
permutation $g$ induces on the first level $\alb\subset\xs$ of the
tree. Define then automorphisms $g|_x\in\autxs$ for $x\in\alb$ by the
rule
\[g(xw)=g(x)g|_x(w)\]
for all $w\in\xs$. It is easy to see that
$g|_x$ are well defined elements of $\autxs$. The map
\[g\mapsto\pi\cdot (g|_x)_{x\in\alb}\in\mathrm{Symm}(\alb)\ltimes\autxs^\alb\]
is an isomorphism $\autxs\arr\mathrm{Symm}(\alb)\ltimes\autxs^\alb$ called the \emph{wreath
  recursion}.
We will identify the elements of $\autxs$ with their images in
$\mathrm{Symm}(\alb)\ltimes\autxs^\alb$ and write $g=\pi\cdot
(g|_x)_{x\in\alb}$.

Consider the case $\alb=\{0, 1\}$. Let us write the
elements of $\autxs^\alb$ as pairs $(h_0, h_1)$, and denote by
$\sigma$ the only non-trivial element of $\mathrm{Symm}(\alb)$. Then
the following equalities uniquely determine elements $g_0, g_1,
g_2\in\autxs$:
\[g_0=\sigma, \quad g_1=(g_0, g_2), \quad g_2=(1, g_1),\]
where an element of either $\mathrm{Symm}(\alb)$ or $\autxs^\alb$ is
not written if it is trivial.
See Figure~\ref{fig:portraits} for a graphical representations of the automorphisms
$g_0, g_1, g_2$.

\begin{figure}
\centering
\includegraphics{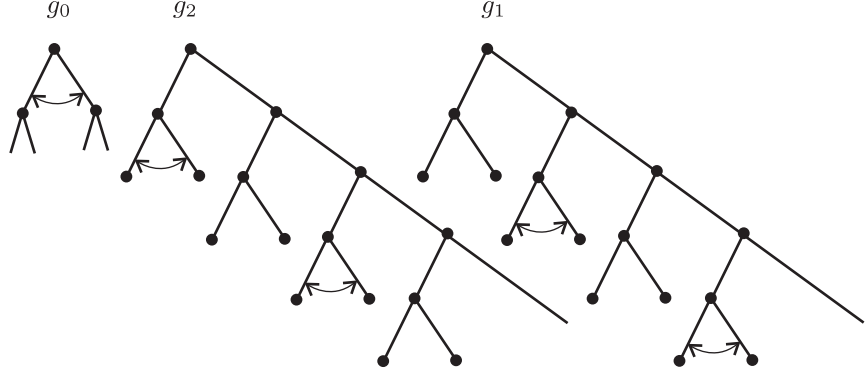}
\caption{Automorphisms $g_0$, $g_1$, $g_2$}
\label{fig:portraits}
\end{figure}

Let us modify the definition of the elements $g_0,
g_1, g_2$ by using either $g_2=(1, g_1)$ or $g_2=(g_1, 1)$ depending
on the level of the tree. More precisely, consider the
space $\xo$ of the right-infinite sequences $x_1x_2\ldots$ over the
alphabet $\alb=\{0, 1\}$ together with the shift map
$s(x_1x_2\ldots)=x_2x_3\ldots$, and define for every $w\in\xo$ the
automorphisms $\alpha_w, \beta_w, \gamma_w$ by the rules
\begin{eqnarray*}
\alpha_w &=& \sigma,\\
\beta_w &=& (\alpha_{s(w)}, \gamma_{s(w)}),\\
\gamma_w &=& \left\{\begin{array}{ll} (\beta_{s(w)}, 1) & \text{if the first
    letter of $w$ is $0$,}\\ (1, \beta_{s(w)}) & \text{if the first
    letter of $w$ is $1$.}\end{array}\right.
\end{eqnarray*}
Denote $\gw_w=\langle\alpha_w, \beta_w, \gamma_w\rangle$. It is easy
to see (e.g., using the criterion from~\cite{gns:conj2}) that $\beta_w$ and
$\gamma_w$ are (independently) conjugate to $g_1$ and $g_2$ (and, of
course, $\alpha_w=g_0$).

The following is shown in~\cite[Proposition~7.1]{nek:ssfamilies}.

\begin{proposition}
For every sequence of polynomials $f_1, f_2, \ldots$ as
in Subsection~\ref{ss:polynomials} there exists a sequence $w\in\xo$ such that
$\img{f_1, f_2, \ldots}$ is conjugate (as a group acting on a binary
rooted tree) to $\gw_w$. Conversely, for every sequence $w\in\xo$ there
exists a sequence of polynomials $f_1, f_2, \ldots$, satisfying the
conditions of~\ref{ss:polynomials} and such that $\img{f_1, f_2,
  \ldots}$ is conjugate to $\gw_w$.
\end{proposition}

The idea of the proof is as follows. Connect the points $\A_0, \B_0,
\G_0$ by disjoint simple paths $l_{\A_0}, l_{\B_0}, l_{\G_0}$ to
infinity. Then define inductively paths $l_{\A_n}, l_{\B_n}, l_{\G_n}$
as lifts of the paths $l_{\B_{n-1}}$ and $l_{\G_{n-1}}$ connecting
$\A_n, \B_n, \G_n$ to infinity. The lifts of the path $l_{\A_{n-1}}$
separate the plane into two connected components, which we label by
$S_0$ and $S_1$, so that the first one contains $\A_n$, and the second one
contains $\G_n$, see Figure~\ref{fig:spiders}.
Then the vertices of the tree $T_t=\bigcup_{n\ge
  0}(f_1\circ\cdots\circ f_n)^{-1}(t)$ are labeled by words over $\{0,
1\}$ according to the itinerary with respect to the defined
partitions. Define the generators $\alpha, \beta, \gamma$ by small
simple loops around $\A_0, \B_0, \G_0$ connected to the basepoint by
paths disjoint with $l_{\A_0}, l_{\B_0}, l_{\G_0}$. It is easy to see
then that $\alpha=\alpha_w, \beta=\beta_w, \gamma=\gamma_w$, where $w$
is the sequence recording in which halves ($S_0$ or $S_1$) of the plane the points
$\B_n$ lie.

\begin{figure}
\centering
\includegraphics{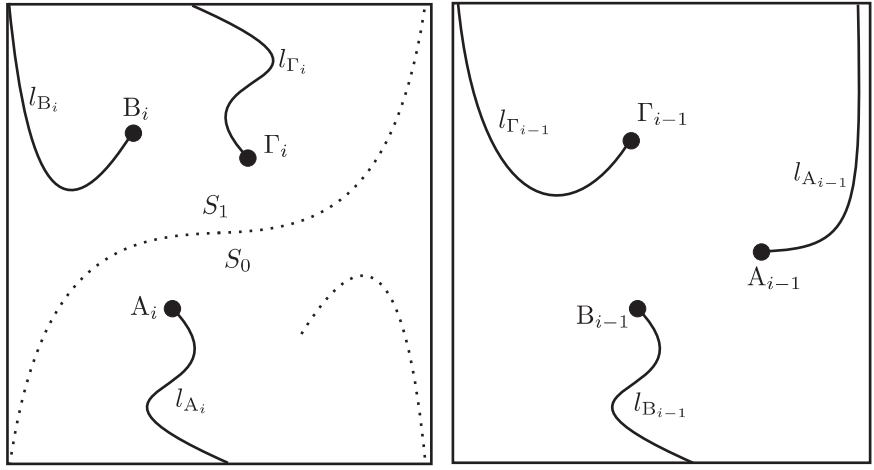}
\caption{Computation of $\img{f_1, f_2, \ldots}$}
\label{fig:spiders}
\end{figure}

In the other direction, for every given $w\in\xo$ 
it is easy to construct a sequence of
branched coverings $f_1, f_2, \ldots$ of plane such that $\img{f_1, f_2,
  \ldots}=\langle\alpha_w, \beta_w, \gamma_w\rangle$ (using the above
construction with paths). Then we can put a complex structure on the
first plane, and pull it back by the maps $f_1\circ\cdots\circ
f_n$. Then $f_n$ will become polynomials satisfying the conditions of
Subsection~\ref{ss:polynomials}.

\begin{proposition}
\label{pr:nonconjugate}
Let $w_1\ne w_2$ be elements of $\xo$. Then there does not exist
$g\in\autxs$ such that $\alpha_{w_1}=g^{-1}\alpha_{w_2}g$,
$\beta_{w_1}=g^{-1}\beta_{w_2}g$, and
$\gamma_{w_1}=g^{-1}\gamma_{w_2}g$.
\end{proposition}

\begin{proof}
Suppose that $w_1$ starts with 0 and $w_2$ starts with 1.
Then $\gamma_{w_1}=(\beta_{s(w_1)}, 1)$, and $\gamma_{w_2}=(1,
\beta_{s(w_2)})$, therefore the conjugator $g$ must permute the
vertices of the first level. On the other hand,
$\beta_{w_1}=(\alpha_{s(w_1)}, \gamma_{s(w_1)})$ and
$\beta_{w_2}=(\alpha_{s(w_2)}, \gamma_{s(w_2)})$, and since
$\alpha_{s(w_1)}$ is not conjugate to $\gamma_{s(w_2)}$, we get a
contradiction.

Suppose that $w_1$ and $w_2$ start with the same letter, and let $g$
be as in the proposition. As in the previous paragraph, considering
$\beta_{w_1}$ and $\beta_{w_2}$ we conclude that $g$ acts trivially on
the first level, i.e., is of the form $g=(h_0, h_1)$ for some $h_0,
h_1\in\autxs$. The equality $\alpha_{w_1}=g^{-1}\alpha_{w_2}g$ implies
that $h_0=h_1$. It follows then that $h_0$ conjugates the triples
$(\alpha_{s(w_1)}, \beta_{s(w_1)}, \gamma_{s(w_1)})$ and
$(\alpha_{s(w_2)}, \beta_{s(w_2)}, \gamma_{s(w_2)})$, which will
eventually lead us to a contradiction.
\end{proof}

In fact, it is proved in~\cite[Proposition~3.1]{nek:ssfamilies} that
for any triple $(h_0, h_1, h_2)$ such that $h_0, h_1, h_2$ are
(independently) conjugate to $g_0, g_1, g_2$, respectively, there exists a
unique sequence $w\in\xo$ such that $(h_0, h_1, h_2)$ are
simultaneously conjugate to $(\alpha_w, \beta_w, \gamma_w)$. We proved
uniqueness in Proposition~\ref{pr:nonconjugate}.

\begin{corollary}
\label{cor:conjclasses}
For every given $w_0\in\xo$ the set of sequences $w\in\xo$ such that
the group $\gw_w$ is conjugate in
$\autxs$ to $\gw_{w_0}$ is at most countable.
\end{corollary}

\begin{proof}
If $\gw_w$ is conjugate to $\gw_{w_0}$, then $\gw_{w_0}$ is generated
by $\alpha'=g^{-1}\alpha_wg$, $\beta'=g^{-1}\beta_wg$, and
$\gamma'=g^{-1}\gamma_wg$ for some $g\in\autxs$. 
By Proposition~\ref{pr:nonconjugate}, the sequence $w$ is uniquely
determined by $\alpha', \beta', \gamma'$. It follows that the
cardinality of the set of possible sequences $w$ is not greater than
the cardinality of the set of generating sets of size 3 of the group
$\gw_{w_0}$, which is at most countable.
\end{proof}

\subsection{A two-dimensional rational map}
\label{ss:ratmap}
Let $f_1, f_2, \ldots$, and $\A_n, \B_n, \G_n$ satisfy the conditions
of Subsection~\ref{ss:polynomials}.
Applying affine transformations $u_n:\C\arr\C$, and
replacing $f_n$ by $u_{n-1}^{-1}f_nu_n$, we may assume that $\A_n=0$,
$\B_n=1$. Denote $p_n=\G_n$. Since $0$ is the critical value of $f_n$,
we have $f_n(z)=(az+b)^2$ for some $a, b\in\C$. Since
$f_n(\A_n)=\B_{n-1}$, we have $b^2=1$, and we may assume that
$b=1$. Since $f_n(\G_n)=\B_{n-1}$, we get $(ap_n+1)^2=1$, hence
$ap_n+1=-1$ (since $\G_n\ne\A_n$). It follows that $a=-2/p_n$. Then
the condition $f_n(\B_n)=\G_{n-1}$ implies that
$p_{n-1}=(1-2/p_n)^2$. (These computations are from~\cite{bartnek:rabbit},
which is the main origin of the family of groups $\gw_w$.)

We arrive naturally to a map $F:\C^2\arr\C^2$ given by
\[F(z, p)=\left(\left(1-\frac{2z}{p}\right)^2,
  \left(1-\frac{2}{p}\right)^2\right).\]
For every sequence $p_n$ such that $p_{n-1}=(1-2/p_n)^2$, the sequence
of polynomials $f_n(z)=(1-2z/p_n)^2$ satisfies the conditions
of~\ref{ss:polynomials}, and conversely, every sequence satisfying the
conditions of~\ref{ss:polynomials} is conjugate to a sequence of this form.

The map $F$ has a natural extension to $\mathbb{CP}^2$ given by
\[F:[z:p:u]\mapsto [(p-2z)^2:(p-2u)^2:p^2].\]
The extended map $F$ is \emph{post-critically finite}: the union $P_F$
of the forward orbits of the critical values of $F$ is an algebraic
variety. It is equal to the union of the lines $z=0, p=0, u=0, z=p,
z=u, p=u$. It is interesting that the map $F$ appeared long ago
in~\cite{fornsibon:crfin} as an example of a post-critically finite
endomorphism of $\mathbb{CP}^2$.

Denote by $\M$ the complement of $P_F$ in
$\mathbb{CP}^2$. Then $F:F^{-1}(\M)\arr\M$ is a covering map of
topological degree 4, and $F^{-1}(\M)\subset\M$. Thus, we can define
its iterated monodromy group $\img{F}$. It is the quotient of the
fundamental group $\pi_1(\M, t)$ by the kernel of its natural action
on the tree $T_F=\bigsqcup_{n\ge 0}F^{-n}(t)$, where $t=(z_0, p_0)$ is
a basepoint.

The iterated monodromy group $\img{F}$ was studied
in~\cite{nek:dendrites}. It is shown there that a natural index two
extension $\wt\gw$ of $\img{F}$ (coming from considering complex conjugation as
a symmetry of $F$) wich is generated by the following six automorphisms of
the rooted tree $\{1, 2, 3, 4\}^*$:
\begin{alignat*}{2}
\alpha &= \sigma, & \quad a &= \pi,\\
\beta &= (\alpha, \gamma, \alpha, \gamma), & \quad b&=(a\alpha,
a\alpha, c, c),\\
\gamma &= (\beta, 1, 1, \beta), &\quad c&=(b\beta, b\beta, b, b),
\end{alignat*} 
where $\sigma=(12)(34), \pi=(13)(24)\in\mathrm{Symm}(\{1, 2, 3, 4\})$.
The iterated monodromy group $\img{F}$ is generated by $\alpha, \beta,
\gamma, ab, ac$.

The group $\wt\gw$ and its index two subgroup $\img{F}$ was used
in~\cite{nek:dendrites} to construct a combinatorial model of the Julia
set of $F$. We will not use the fact that $\wt\gw$ is an index two
extension of $\img{F}$. For us it is just the group generated by the
elements $\alpha, \beta, \gamma, a, b, c$ defined by the above recurrent
relations.

\subsection{Further properties of the groups $\gw_w$}
Let $\wt\gw=\langle\alpha, \beta, \gamma, a, b, c\rangle$, as 
defined in the previous subsection. Denote $\gw=\langle\alpha, \beta,
\gamma\rangle<\wt\gw$. It is the subgroup of $\img{F}$ defined by
loops inside the hyperplane $p=p_0$ (where $p_0$ is the second
coordinate of the basepoint $t=(z_0, p_0)\in\C^2$).

Consider the map $P:\{1, 2, 3, 4\}^*\arr\xs$ that applies to
each letter of a word $v\in\{1, 2, 3, 4\}^*$ the map
\[1\mapsto 0,\quad 2\mapsto 0,\quad 3\mapsto 1,\quad 4\mapsto 1.\]
It is easy to check that the map $P$ defines an imprimitivity system for
the group $\wt\gw$, i.e., that there exists an action of $\wt\gw$ on $\xs$
such that $g(P(v))=P(g(v))$ for all $v\in\{1, 2, 3, 4\}^*$. The action
of $\wt\gw$ on $\xs$ is uniquely determined by this condition.

Note that it follows directly from the recursive definition of the
generators $\alpha, \beta, \gamma, a, b, c$, that
$\gw$ belongs to the kernel of the action of $\wt\gw$ on
$\xs$, and that the action of $a, b, c$ on $\xs$ are defined by the
recurrent rules
\begin{equation}\label{eq:abc}
a=\sigma,\quad b=(a, c),\quad c=(b, b).
\end{equation}
Let $H$ be the quotient of $\wt\gw$ by the kernel of its action on
$\xs$.

Direct computations show that the following relations hold in
$\wt\gw$.
\begin{alignat*}{3}
\alpha^a&=\alpha, &\quad\alpha^b&=\alpha, &\quad \alpha^c&=\alpha,\\
\beta^a&=\beta, &\quad\beta^b&=\beta, &\quad\beta^c&=\beta^\gamma,\\
\gamma^a&=\gamma^\alpha, &\quad\gamma^b&=\gamma^\beta,&\quad\gamma^c&=\gamma,
\end{alignat*}
which implies that $\gw$ is a normal subgroup of $\wt\gw$
(see~\cite[Proposition~4.4]{nek:ssfamilies}). In fact, it is proved
in~\cite[Proposition~4.8]{nek:ssfamilies} (but we will not need it in
our paper) that $\gw$ coincides with
the kernel of the action of $\wt\gw$ on $\xs$, i.e., that $\wt\gw/\gw$
is the group $H$ generated by the elements $a, b, c$ defined
by~\eqref{eq:abc}.

For an infinite sequence $w=x_1x_2\ldots\in\xo$, denote by $T_w$ the
rooted subtree of $\{1, 2, 3, 4\}^*$ equal to the inverse image
$P^{-1}(\{\emptyset, x_1, x_1x_2, x_1x_2x_3, \ldots\})$ of the
corresponding path in $\xs$. It follows directly from the definition of the
generators $\alpha, \beta, \gamma$ of $\gw$:
\[\alpha=(12)(34),\quad\beta=(\alpha, \gamma, \alpha,
\gamma),\quad\gamma=(\beta, 1, 1, \beta),\]
that the tree $T_w$ is $\gw$-invariant, and that restriction of the
action of the generators $\alpha, \beta, \gamma$ onto $T_w$ coincides
with the action of the generators $\alpha_w, \beta_w, \gamma_w$ of
$\gw_w$, if we identify $T_w$ with $\xs$ by the map
\[1\mapsto 0,\quad 2\mapsto 1,\quad 3\mapsto 0,\quad 4\mapsto 1.\]

Consequently, the groups $\gw_w$ are obtained by restricting the
action of the group $\gw$ onto the subtrees $T_w$. We will denote the
canonical epimorphism $\gw\arr\gw_w$ by $P_w$.

For every $g\in\wt\gw$ we have $g(T_w)=T_{g(w)}$, just by the definition of
the action of $\wt\gw$ on $\xo$. Then the next lemma follows from the fact
that $\gw$ is normal in $\wt\gw$.

\begin{lemma}
\label{lem:congw}
For every $g\in\wt\gw$ the isomorphism
$g:T_w\arr T_{g(w)}$ conjugates the groups $\gw_w$ and $\gw_{g(w)}$.
\end{lemma}

The group $\wt\gw$ is transitive on the first level of the tree $\{1,
2, 3, 4\}^*$, and the homomorphism from the stabilizer of the first
level to $\wt\gw$ given by $(g_1, g_2, g_3, g_4)\mapsto g_1$ is
onto. This implies that $\wt\gw$ is transitive on each level of the
tree (it is \emph{level-transitive}), see~\cite[Corollary~2.8.5]{nek:book}.
Consequently $H$ is also level-transitive.

We have finished a sketch of the proof of the following fact (see
also~\cite[Proposition~4.6]{nek:ssfamilies})

\begin{proposition}
\label{pr:conjugate}
If $w_1, w_2\in\xo$ belong to one $H$-orbit, then $\gw_{w_1}$ and
$\gw_{w_2}$ are conjugate.
\end{proposition}

In fact, the converse is also true: if the groups $\gw_{w_1}$ and
$\gw_{w_2}$ are conjugate, then $w_1$ and $w_2$ belong to one
$H$-orbit, see~\cite[Theorem~5.1]{nek:ssfamilies}. 

Note that since the action of $H$ on $\xs$ is level-transitive, for
every $w\in\xo$ the set of sequences $w'\in\xo$ such that $\gw_w$ is
conjugate to $\gw_{w'}$ is dense in $\xo$. It also follows that
closures of $\gw_w$ in the profinite group $\autxs$ do not depend, up
to isomorphism, on $w$.

Recall, that for $g\in\autxs$ and $x\in\alb$ we denote by $g|_x$ the
element of $\autxs$ defined by the condition $g(xu)=g(x)g|_x(u)$. We use
a similar definition of arbitrary words. Namely, for $v\in\xs$ and
$g\in\autxs$, denote by $g|_v$ the automorphism of $\xs$ defined by
the condition that
\[g(vu)=g(v)g|_v(u)\]
for all $u\in\xs$. The element $g|_v$ is called the \emph{section} of
$g$ at $v$, and it can be computed by repeated application of the wreath
recursion, since it satisfies
\[g|_{x_1x_2\ldots x_n}=g|_{x_1}|_{x_2}\ldots|_{x_n}.\]

The following statement is proved by direct computation
(see~\cite[Proposition~4.2]{nek:ssfamilies}).

\begin{proposition}
\label{pr:contracting}
The subgroups $\langle\alpha, \beta\rangle$, $\langle\alpha,
\gamma\rangle$, and $\langle\beta, \gamma\rangle$ of $\gw$ are isomorphic to
dihedral groups of order 16, 8, and 16, respectively. Denote their
union by $\nuke$. For every $g\in\gw$ there exists $n$ such that
$g|_v\in\nuke$ for all words $v$ of length at least $n$.
\end{proposition}

In order to prove the second statement of the proposition one shows
that the sections of elements of $\nuke\cdot\{\alpha, \beta, \gamma\}$ at long
enough words belong to $\nuke$.

The next statement is also checked directly, see~\cite[Proposition~4.3]{nek:ssfamilies}.

\begin{lemma}
The canonical epimorphism $P_w:\gw\arr\gw_w$ is injective on $\nuke$
for all $w$ different from $111\ldots$. 

The image of
$\beta_{111\ldots}\gamma_{111\ldots}$ has order 2 in $\gw_{111\ldots}$.
\end{lemma}

Denote by $\xo_0$ the set of all sequences $w\in\xo$ with infinitely
many zeros.

\begin{proposition}
\label{pr:relations}
Let $w\in\xo_0$, and let $A$ and $B$ be finite subsets of $\gw$ such that
$P_w(g)=1$ for all $g\in A$ and $P_w(g)\ne 1$ for all $g\in B$. Then
there exists $n$ such that if $w'\in\xo_0$ has the same beginning of
length $n$ as $w$, then $P_{w'}(g)=1$ for all $g\in A$ and
$P_{w'}(g)\ne 1$ for all $g\in B$.
\end{proposition}

In other words, the map $w\mapsto\langle\alpha_w, \beta_w,
\gamma_w\rangle$ is a continuous map from $\xo_0$ to the space of
three-generated groups. 

\begin{proof}
It is enough to prove that for every $g\in\gw$ 
there exists $n$ such that if the length of the common beginning of
sequences $w_1, w_2\in\xo_0$ is at least $n$, then 
$P_{w_1}(g)=1$ if and only if $P_{w_2}(g)=1$.

Let us apply the wreath recursion defining $\gw_w$ several
times to $g$. By Proposition~\ref{pr:contracting}, there exists $n$ such that
all sections $g|_v$ belong to $\nuke$ for all words $v$ of length at
least $n$. If the action of $g$ is non-trivial on the $n$th level of
the tree $T_{w_1}$, then $P_{w_1}(g)\ne 1$. Suppose that the action is
trivial. Then $P_{w_1}(g)=1$ if and only if all sections
$P_{w_1}(g)|_{P_{w_1}(v)}=P_{s^n(w_1)}(g|_v)$ for $v\in
T_{w_1}\cap\{1, 2, 3, 4\}^n$ are trivial. But $g|_v\in\nuke$, and if
$s^n(w_1)\ne 111\ldots$, then $g|_v\ne 1$ if and only if
$P_{s^n(w_1)}(g|_v)\ne 1$. It follows that triviality of $P_{w_1}(g)$
depends only on the first $n$ letters of $w_1$, provided $s^n(w_1)\ne
111\ldots$.
\end{proof}

\section{Profinite completion of $\gw_w$}

\begin{defi}
Let $G$ be a group acting on a rooted tree $\xs$. The \emph{$n$th
  level rigid stabilizer} $\rist_n(G)$ is the group generated by
elements $g\in G$ such that $g$ acts trivially on the $n$th level
$\alb^n$ of $\xs$, and the sections $g|_v$ are 
trivial for all words $v\in\alb^n$ except for one.
\end{defi}

Let us denote $\rist_{n, w}=\rist_n(\gw_w)$. 
Denote by $L_w$ the  smallest normal subgroup $L_w$ of $\gw_w$ containing
$\{[\alpha_w, \beta_w], [\gamma_w, \beta_w]\}$. 
Direct computations
(see~\cite[Proposition~5.14]{nek:ssfamilies}) show
that $L_w$ (more formally, its image under the wreath recursion)
contains $L_{s(w)}\times L_{s(w)}$. This implies that the direct product $L_{s^n(w)}^{\alb^n}$
is contained in $\rist_{n, w}$. Moreover, it is also checked directly
that index of $L_w$ in $\gw_w$ is finite. It
follows that $\rist_{n, w}$ is of finite index in $\gw_w$.

\begin{defi}
A group acting on rooted tree is \emph{branch} if $\rist_n(G)$ is of
finite index in $G$ for all $n$. It is called \emph{weakly branch} if
  $\rist_n(G)$ are infinite for all $n$.
\end{defi}

See~\cite{handbook:branch} for more on branch groups.

The following theorem is proved in~\cite{lavnek} (see
also~\cite[Proposition~2.10.7]{nek:book}).

\begin{theorem}
Let $G_1, G_2\le\autxs$ be weakly branch groups, and let $\psi:G_1\arr
G_2$ be an isomorphism. Suppose that there exist subgroups $H_n\le
G_1$ for all $n\ge 1$ such that the following conditions hold
\begin{enumerate}
\item $H_n$ and $\psi(H_n)$ belong to the stabilizer of the $n$th
  level;
\item the groups $H_n$ and $\psi(H_n)$ act level-transitively on the
  rooted subtrees $v\xs\subset\xs$ for $v\in\alb^n$.
\end{enumerate}
Then the isomorphism $\psi$ is induced by conjugation in $\autxs$.
\end{theorem}

It is easy to produce subgroups $H_n$ of $\gw_w$ satisfying the
conditions of the above theorem. For example, one can take $H_1$ to be
the group generated by squares of elements $\gw_w$, and then define
inductively $H_n$ as the group generated by squares of the elements of
$H_{n-1}$, see~\cite[Section~6]{nek:ssfamilies}.

It follows that the isomorphism relation between groups $\gw_w$
coincides with conjugacy. In particular, it follows from
Corollary~\ref{cor:conjclasses} that the isomorphism classes in the
family $\gw_w$ are countable. More precise information
gives~\cite[Theorem~6.1]{nek:ssfamilies}: two groups $\gw_{w_1}$ and
$\gw_{w_2}$ are isomorphic if and only if the sequences $w_1$ and
$w_2$ are co-final, i.e., are of the form $w_1=v_1w$ and $w_2=v_2w$, where
the words $v_1, v_2\in\xs$ have equal length.

The following theorem is proved in~\cite[Theorem~5.2 and
Lemma~5.3]{handbook:branch}.

\begin{theorem}
Let $G$ be a level-transitive subgroup of $\autxs$. For every
non-trivial normal subgroup $N$ of $G$ there exists $n$ such that
$N$ contains the commutator subgroup $\rist_n(G)'$ of $\rist_n(G)$.
\end{theorem}

It is not hard to show that $\rist_{n, w}'$ has finite index in
$\gw_w$ (see~\cite[Proposition~5.15]{nek:ssfamilies}). In particular,
this shows that the groups $\gw_w$ are just-infinite, i.e., that all
their non-trivial normal subgroups have finite index. (Note that there
is a typo in~\cite[Proposition~5.15]{nek:ssfamilies}: the word
``normal'' is missing.)

\begin{theorem}
\label{th:main}
The profinite completion $\widehat{\gw_w}$ 
of $\gw_w$ does not depend on $w$, if $w$ has
infinitely many zeros.
\end{theorem}

\begin{proof}
By~\cite[Corollary~3.2.8]{ribeszalesskii:book}, it is enough to prove that
the sets of all finite quotients of $\gw_w$ do not depend on
$w\in\xo_0$.
Since a group is a proper quotient of $\gw_w$ if and
only if it is a quotient of $\gw_w/\rist_{n, w}'$ for some $n$, it is enough to
prove Proposition~\ref{pr:gwrist} below.
\end{proof}

\begin{proposition}
\label{pr:gwrist}
The quotient $\gw_w/\rist_{n, w}'$ does not depend, up to
isomorphism, on $w\in\xo_0$.
\end{proposition}

\begin{proof} Fix a sequence $w\in\xo_0$.
The subgroups $\rist_{n, w}$ and $\rist_{n, w}'$ have
  finite index in $\gw_w$, hence they are finitely generated.

Let $[P_w(g_1), P_w(h_1)], [P_w(g_2), P_w(h_2)], \ldots, [P_w(g_k),
P_w(h_k)]$ be a finite
generating set of $\rist_{n, w}'$, where $g_i, h_i\in\gw$ are such
that $P_w(g_i), P_w(h_i)\in\rist_{n,
  w}$. Moreover, we may assume that for every $i$ 
all sections of $P_w(g_i)$ and $P_w(h_i)$ at words of length $n$ are trivial 
except for one word $v_i$. It follows from
Proposition~\ref{pr:relations} that there exists $m_1\ge n$ such that if
$w'\in\xo_0$ is such that $w$ and $w'$ have a common beginning of
length at least $m_1$, then $P_{w'}(g_i), P_{w'}(h_i)\in\rist_{n, w'}$
for all $i=1, \ldots, k$. It follows that $P_{w'}([g_i,
h_i])\in\rist_{n, w'}'$ for all $i=1, \ldots, k$.

Let $A=\{1, a_1, a_2, \ldots, a_m\}\subset\gw$ be such that $P_w(A)$ is
a coset transversal
of $\gw_w$ modulo $\rist_{n, w}'$. Consider the
multiplication table for $\gw_w/\rist_{n, w}'$, and write its entries as
relations of the form $P_w(a_ia_ja_k^{-1}r_{i, j})=1$, where $r_{i,
  j}$ is a product of
the generators $P_w([g_i, h_i])$ of $\rist_{n, w}'$. 
We get a finite set of
relations, hence, by Proposition~\ref{pr:relations}, there exists
$m_2\ge m_1$
such that if $w'\in\xo_0$ has a common beginning with $w$ of length at
least $m_2$, then the corresponding
relations $P_{w'}(a_ia_ja_k^{-1}r)$ also hold in $\gw_{w'}$. It follows that
$\gw_{w'}/\rist_{n, w'}'$ is a homomorphic image of $\gw_w/\rist_{n,
  w}$, where the homomorphism is induced by the map
\begin{equation}\label{eq:Pw}
P_w(g)\mapsto P_{w'}(g),\qquad g\in A.
\end{equation}

Choose now $w\in\xo_0$ such that $\gw_w/\rist_{n, w}'$ has the smallest
possible order. Then for all $w'$ with a long enough common beginning
with $w$ the group
$\gw_{w'}/\rist_{n, w'}'$ is isomorphic to $\gw_w/\rist_{n, w}'$. But the
orbits of $H$ on $\xo$ are dense, hence for every $w'\in\xo_0$
there exists $g\in\wt\gw$ such that $g(w')$ is arbitrarily close to
$w$, i.e., has arbitrarily long common beginning with $w$.
By Lemma~\ref{lem:congw}, $\gw_{w'}/\rist_{n, w'}'$ is isomorphic to
$\gw_{g(w')}/\rist_{n, g(w')}'$, which finishes the proof.
\end{proof}

\section{Word growth}

Let $G$ be a finitely generated group. For a finite generating set
$S$, the \emph{word growth function} if the function $r_{G, S}(n)$
equal to the number of elements of $G$ that can be
represented as products of length at most $n$ of the generators $s\in
S$ and their inverses. 

For two non-decreasing positive functions $r_1,
r_2:\mathbb{N}\arr\mathbb{R}$, we write $r_1\prec r_2$ if there exists
a constant $C>1$ such that \[r_1(n)\le Cr_2(C(n))\] for all
$n\in\mathbb{N}$. We say that $r_1$ and $r_2$ are \emph{equivalent} if
$r_1\prec r_2$ and $r_2\prec r_1$.

The \emph{growth type} of a finite generated group $G$ is the
equivalence class of its word growth function. It is easy to prove that the
equivalence class of the growth function $r_{G, S}(n)$ does not
depend on the choice of the generating set $S$. See~\cite{mann:growth}
for a survey of results on word growth of groups.

It is a consequence of Gromov's theorem on
groups of polynomial growth~\cite{gro:gr} and the formula for degree
of polynomial growth of a virtually nilpotent group~\cite{bass:nilpotent} 
that if $G_1$ and $G_2$ are finitely generated residually finite
groups such that $\widehat{G_1}\cong\widehat{G_2}$, and $G_1$ is of
polynomial growth, then growth types of $G_1$ and $G_2$ are the
same. In other words, the growth type of groups of polynomial growth
is a profinite property. In general, however, we can not reconstruct
the growth type of a group from its profinite completion.

\begin{proposition}
The set of growth types of groups $\{\gw_w\;:\;w\in\xo_0\}$ is
uncountable. In particular, growth type of a finitely generated
residually finite group is not a profinite property.
\end{proposition}

The first statement of the proposition
was proved in~\cite[Corollary~4.7]{nek:nonuniform}. The proof
uses the fact that the closure of the set $\{\gw_w\;:\;w\in\xo\}$ in
the space of three-generated groups contains a group of exponential
and a group of intermediate growth. The latter is 
$\gw_{000\ldots}\cong\img{z^2+i}$, sub-exponential growth of which was
proved in~\cite{buxperez:imgi}. One can use these facts and
Proposition~\ref{pr:relations} to prove, in the same way as it is done
for the family of Grigorchuk groups in~\cite{grigorchuk:growth_en}, that
the set of growth types of groups $\gw_w$ is uncountable.

\providecommand{\bysame}{\leavevmode\hbox to3em{\hrulefill}\thinspace}
\providecommand{\MR}{\relax\ifhmode\unskip\space\fi MR }
\providecommand{\MRhref}[2]{%
  \href{http://www.ams.org/mathscinet-getitem?mr=#1}{#2}
}
\providecommand{\href}[2]{#2}

\end{document}